\documentclass[a4paper,12pt]{amsart}

\usepackage{amsmath,amsthm,amsfonts,amssymb,stmaryrd,mathrsfs,enumitem}
\usepackage{latexsym}
\usepackage{fullpage}
\usepackage{a4,color,palatino,fancyhdr}
\usepackage{graphicx}
\usepackage{float}  
\usepackage[small, bf, margin=90pt, tableposition=bottom]{caption}
\RequirePackage{amssymb}
\RequirePackage[T1]{fontenc}
\usepackage{amscd}
\usepackage{xypic}
\input{xy}
\xyoption{all}
\xyoption{poly}
\usepackage[all]{xy}

\newtheorem{thm}{Theorem}[section]
\newtheorem{prop}[thm]{Proposition}
\newtheorem{defnm}[thm]{Definition}
\newtheorem{lem}[thm]{Lemma}
\newtheorem{coro}[thm]{Corollary}
\newtheorem{rem}{Remark}

\newtheorem{ques}[thm]{Question}
\newtheorem{conj}[thm]{Conjecture}

\numberwithin{equation}{section}

\newcommand\Z{{\mathbb{Z}}}
\newcommand\R{{\mathbb{R}}}
\newcommand\C{{\mathbb{C}}}

\begin{document}

\title{Equivariant cohomology Chern numbers determine equivariant unitary bordism for torus groups}
\author{Zhi L\"u and Wei Wang}

\thanks{Supported by grants from NSFC (No. 11371093, No. 11301335, No. 11431009 and No. 11661131004)}
\address{School of Mathematical Sciences, Fudan University, Shanghai,
200433, P.R. China.} \email{zlu@fudan.edu.cn}
\address{College of Information Technology, Shanghai Ocean University, 999 Hucheng Huan Road, 201306, Shanghai, P.R.  China.}
\email{weiwang@amss.ac.cn}

\date{}
\begin{abstract}
This paper shows that the integral equivariant cohomology Chern numbers  completely determine the equivariant geometric unitary bordism classes of closed unitary $G$-manifolds, which gives an affirmative
answer to the conjecture posed by Guillemin--Ginzburg--Karshon in ~\cite[Remark H.5, \S3,  Appendix H]{Gullemin}, where $G$ is a torus. As a further application, we also obtain a satisfactory solution of~\cite[Question (A), \S1.1, Appendix H]{Gullemin} on unitary Hamiltonian $G$-manifolds. Our key ingredients in the proof are  the universal toric genus defined by Buchstaber--Panov--Ray and the Kronecker pairing of bordism and cobordism.
Our approach heavily  exploits  Quillen's geometric interpretation of  homotopic unitary cobordism theory.
 Moreover, this method can also be applied to the study of $({\Bbb Z}_2)^k$-equivariant unoriented bordism and can still derive the classical result of tom Dieck.
\end{abstract}
\maketitle

\section{Introduction and main results}

\subsection{Background} In his seminal work~\cite{Thom} , R. Thom introduced the unoriented bordism theory, which corresponds to the infinite orthogonal group $O$.  Since then, various other bordism theories, which correspond to subgroups $\mathbb{G}$ of the orthogonal group $O$ as structure groups of stable tangent bundles or stable normal bundles of compact smooth manifolds, have been studied and established (e.g., see~\cite{Wall, M, N} and for more details, see~\cite{L, Stong}). When $\mathbb{G}$ is chosen as $SO$ (resp. $U, SU$ etc.), the corresponding bordism theory is often called the oriented (resp. unitary, special unitary etc.) bordism theory.
One of the main results in these bordism theories  is that a closed manifold bounds if and only if certain characteristic numbers vanish.
For example, when $\mathbb{G}$ is $O$ or $U$, the corresponding characteristic numbers will be Stiefel--Whitney numbers or Chern numbers. Another main result is that the bordism ring $\Omega_*^{\mathbb{G}}$ associated to the given subgroup $\mathbb{G}$ can also be described quite completely, where $\Omega_*^{\mathbb{G}}$ consists of the bordism classes of all closed smooth manifolds with $\mathbb{G}$-structure (i.e., whose stable tangent bundles admit  $\mathbb{G}$ as structure group). If $\mathbb{G}$ is $O$ (resp. $SO, U$ etc.), then a compact smooth manifold with $\mathbb{G}$-structure  is often called an unoriented (resp. oriented, unitary etc.) manifold.
 Homotopy theoretic  bordism theories $  M\mathbb{G}_*(X)(\cong \Omega_*^{\mathbb{G}}(X)$ by the Pontryagin--Thom construction) and  cobordism theories $M\mathbb{G}^*(X)$ of a topological space $X$ are due to M. Atiyah~\cite{A}, where $M\mathbb{G}$ is the corresponding Thom spectrum of $\mathbb{G}$ and $\Omega_*^{\mathbb{G}}(X)$ is the geometric bordism ring generated by the bordism classes of singular manifolds $f:M\longrightarrow X$ where each $M$ is a smooth closed manifold with $\mathbb{G}$-structure. Note that when $X$ is a point, $\Omega_*^{\mathbb{G}}(X)$ is just $\Omega_*^{\mathbb{G}}$.  These two kinds of theories are generalized homology and cohomology theories.

  \

 In the early 1960s, Conner and Floyd (\cite{cf1, cf2, cf3}) began the study of geometric equivariant bordism theory by combining the ideas of bordism theory and transformation groups. They studied  geometric equivariant unoriented,  oriented and unitary bordism theories for  smooth closed manifolds with periodic diffeomorphisms, and the  subject has also continued to further develop  by extending their ideas or  combining with various  ideas and theories from other research areas  since then. For example,  tom Dieck~\cite{tom Dieck1} introduced and studied the homotopy theoretic equivariant bordisms $M\mathbb{G}_*^G(X)$ and cobordisms $M\mathbb{G}^*_G(X)$ for a $G$-space $X$ where $G$ is a compact Lie group, and recently, Buchstaber--Panov--Ray in \cite{Buch} introduced and studied the  universal toric genus  in a geometric manner.
 The geometric equivariant bordism ring $\Omega_*^{\mathbb{G}, G}(X)$ for a $G$-space $X$ can also be defined in a natural way, and it is formed by singular $G$-manifolds $f: M\longrightarrow X$ where each $M$ is a smooth closed $G$-manifold with $\mathbb{G}$-structure such that the $G$-action preserves the $\mathbb{G}$-structure (in this case, we say that $M$ admits a $G$-equivariant $\mathbb{G}$-structure). When $X$ is a point, $\Omega_*^{\mathbb{G}, G}(X)$ will be denoted by $\Omega_*^{\mathbb{G}, G}$, and it is an equivariant analogue of $\Omega_*^{\mathbb{G}}$. It should be pointed out that in the above definition, the $G$-equivariant $\mathbb{G}$-structure of a smooth closed $G$-manifold with $\mathbb{G}$-structure is equipped on the stably $G$-equivariant tangent bundle of $M$ rather than the  stably $G$-equivariant normal bundle of
the embedding of $M$ in some $G$-representation because there is a substantial difference between both notions (see~\cite{Hanke}). In this paper we will pay more attention to the case in which $\mathbb{G}=U$ (i.e.,  $G$-unitary bordism).   A nice survey of the development of $G$-unitary bordism is contained in the introduction of the paper \cite{Buch} by Buchstaber--Panov--Ray.

\

Comparing with the nonequivariant case, a natural question arises and  is stated as follows:
\begin{ques}  What kinds of equivariant characteristic numbers determine the
equivariant geometric bordism class of a smooth closed $G$-manifold with $\mathbb{G}$-structure?
\end{ques}

\

As far as the authors know, the above question is far from solved. Most known works with respect to the above question have considered mainly the cases $(\mathbb{G}, G)=(O, (\Z_2)^k)$, $(U, T^k\times\Z_m)$, $(SO, \text{finite group})$ (e.g., see~\cite{BT, Gullemin,  H, K,  LW, LT2, S, tom Dieck2, tom Dieck3}). In the following, we mainly give an investigation in the cases $\mathbb{G}=O, U$.

\

tom Dieck first investigated the above question.  In his series of papers (\cite{tom Dieck1, tom Dieck2, tom Dieck3}), by combining the geometric approach of Conner--Floyd (\cite{cf2, cf3}) and the K-theory approach developed by Atiyah, Bott, Segal and Singer (\cite{AB, AS1, AS2}), tom Dieck  introduced the bundling transformation which further develops  ideas of Boardman~\cite{B} and Conner~\cite{Co},   and proved a series of integrality theorems in equivariant (co)bordism theory. Furthermore, he gave an answer to
  the above question in the cases $(\mathbb{G}, G)=(O, (\Z_2)^k)$, $(U, T^k\times\Z_m)$, which is stated as  follows:
 \begin{thm}[tom Dieck]
 \label{tom}
 When $(\mathbb{G}, G)=(O, (\Z_2)^k)$, $\beta\in \Omega_*^{O,(\Z_2)^k}$ is zero if and only if all  equivariant Stiefel--Whitney numbers of $\beta$ vanish.
  When $(\mathbb{G}, G)=(U, T^k\times\Z_m)$,
 $\beta\in \Omega_*^{U,G}$ is zero if and only if all  equivariant K-theoretic  Chern numbers of $\beta$ vanish.
\end{thm}
In their book~\cite[Appendix H]{Gullemin}, Guillemin--Ginzburg--Karshon discussed the problem of calculating the ring $\mathcal{H}_{*}^{G}$ of equivariant Hamiltonian bordism classes of all unitary Hamiltonian $G$-manifolds with integral equivariant cohomology classes ${1\over{2\pi}}[\omega-\Phi]$, where $G$ is a torus.   With respect to the  determination of  the ring $\mathcal{H}_{*}^{G}$,
they posed three series of questions, the first one of which is stated as follows:

 \begin{ques}[{\cite[Question (A), \S1.1, Appendix H]{Gullemin}}] \label{qu} Do  mixed equivariant characteristic numbers form a full system of invariants of Hamiltonian bordism?
 \end{ques}

 On this question, Guillemin--Ginzburg--Karshon constructed a monomorphism
 \begin{equation*}
\Sigma_{G}:\mathcal{H}_{*}^{G}\longrightarrow \Omega_{*+2}^{U, G},
\end{equation*}
so that Question~\ref{qu} is equivalent to asking if
the integral equivariant cohomology Chern numbers determine the equivariant geometric unitary bordism classes
for the ring  $\Omega_*^{U,G}$.  They showed that a closed unitary $G$-manifold $M$ with only isolated fixed-points represents the zero element in $\Omega_*^{U,G}$ if and only if all integral equivariant cohomology Chern numbers of $M$ vanish, which gives a partial solution of  Question~\ref{qu}. Furthermore, they posed the following conjecture without the restriction of isolated fixed-points.
\begin{conj}[{\cite[Remark H.5, \S3,  Appendix H]{Gullemin}}]\label{Conj}
Let $G$ be a torus. Then  $\beta\in \Omega_*^{U,G}$ is zero if and only if all
integral equivariant cohomology Chern  numbers of $\beta$ vanish.
\end{conj}

For the detecting problem, it is natural to analyze the topological data of fixed point sets. This approach has been used successfully in the case when the fixed point sets are isolated (\cite{Gullemin}).  For the general case, we found that the analysis of fixed data seems to be not easy to handle and we turned to consider the geometric presentation of the image $\Phi([M])\in MU_*(BG)$ of {\em the universal toric genus} instead, which admits a natural geometric representation described by Buchstaber--Panov--Ray in \cite{Buch}. In particular, by using the Kronecker pairing, our argument can directly be reduced to the simpler calculation, so that the Boardman map as used in the tom Dieck's work (\cite{tom Dieck2, tom Dieck3}) will not be involved yet.

\subsection{Main results}
The motivation of this paper is mainly supplied by the work of Guillemin--Ginzburg--Karshon  as mentioned above.
The geometric description of the universal toric genus developed by Buchstaber--Panov--Ray in \cite{Buch} provides us much more insight, so that we can carry out our work by utilizing the Kronecker pairing of bordism and cobordism. This leads us to give an affirmative
answer to  Conjecture~\ref{Conj}, which is stated as follows.
\begin{thm}\label{main1}
Conjecture~\ref{Conj} holds.
\end{thm}

As a further consequence, we obtain a satisfactory solution of Question~\ref{qu}.
\begin{coro} Mixed equivariant characteristic numbers determine equivariant Hamiltonian  bordism.
\end{coro}
It is interesting that our approach above can also be applied to the detecting problem of $({\Bbb Z}_2)^k$-equivariant unoriented bordism, and in particular, it can still derive the classical result of tom Dieck in the case $(\mathbb{G}, G)=(O, (\Z_2)^k)$ of Theorem~\ref{tom} by replacing the Boardman map by Kronecker pairing in the original proof of  tom Dieck~\cite{tom Dieck2}.

\

This paper is organized as follows. In Section~\ref{prelim}, we shall review Quillen's
geometric interpretation of homotopic unitary cobordism theory first, which is of essential importance, so that the Kronecker pairing between bordism and cobordism can be calculated in a geometric way. We then review the universal toric genus defined by Buchstaber--Panov--Ray, which can be expressed in terms of Quillen's
geometric interpretation. We also recollect some necessary facts about equivariant Chern classes and equivariant Chern numbers. Then we shall give the proof of Theorem~\ref{main1} in Section~\ref{proof}, which is more geometric. In particular, our approach in the unitary case can  also be  carried out in the study of $({\Bbb Z}_2)^k$-equivariant unoriented bordism, as we shall see  in the final part of Section~\ref{proof}.

\subsection*{Acknowledgements}
 The authors would like to thank Peter Landweber for his comments and suggestions on an earlier version of this paper. We are also grateful to the referee for carefully reading this manuscript and providing a number of valuable thoughts and helpful suggestions, which led to this version.

\section{Preliminaries}\label{prelim}
\subsection{Geometric interpretation of elements in $MU^*(X)$ and Kronecker pairing}\label{geo}

Given a topological space $X$, let $MU_*(X)$ and $MU^*(X)$ be the  complex (homotopic) bordism and cobordism of $X$,  which are defined as
$$MU_*(X)=\lim_{l\longrightarrow\infty} [S^{2l+*},  X_+\wedge MU(l)]$$
and
$$MU^*(X)=\lim_{l\longrightarrow\infty}[S^{2l-*}\wedge X_+, MU(l)]$$
respectively, where $X_+$ denotes the union of $X$ and a disjoint point, and $MU(l)$ denotes the Thom space of the universal complex $l$-dimensional vector bundle over $BU(l)$.
 \vskip .1cm
Geometrically, it is very well-known that elements of $MU_n(X)$ are  regarded as bordism classes of
maps $M\longrightarrow X$ of stably complex  closed $n$-dimensional manifolds $M$ to $X$ since
$MU_n(X)\cong \Omega_n^U(X)$. 

On the geometric interpretation of elements in $MU^*(X)$, Quillen showed in
\cite[Proposition 1.2]{Q}  that for a manifold $X$, $MU^{\pm n}(X)$ is isomorphic to the group formed by cobordism classes of proper {\em complex-oriented maps} $f: M \longrightarrow X$ of dimension $\mp n$, where $\dim X-\dim M=\pm n$.  Recall that when $n$ is even, a {\em complex-oriented map} $f: M \longrightarrow X$ of dimension $\mp n$ is a composition od maps of manifolds
\begin{equation}\label{com-ori}
\begin{CD}
M @ >i>>\mathbb{ E} @>\pi>> X
\end{CD}
\end{equation}
where $\pi: \mathbb{E}\longrightarrow X$  is a complex vector bundle over $X$, and $i: M\longrightarrow \mathbb{E}$  is an embedding  endowed with a complex structure on its normal bundle.
When $n$ is odd, the complex orientation of $f$ will be defined in a same way as above with $\mathbb{E}$ replaced by $\mathbb{E}\times \R$ in~(\ref{com-ori}).
 \vskip .1cm
It is very well-known that as generalized cohomology and homology theory, $MU^*(X)$ and $MU_*(X)$ admit a canonical Kronecker pairing
$$\langle\ ,\ \rangle: MU^n(X)\otimes MU_m(X)\longrightarrow MU_{m-n}$$
where $MU_{m-n}=MU_{m-n}(pt)\cong \Omega_{m-n}^U$.

 \vskip .1cm
This Kronecker pairing can be calculated in a geometric way as follows. For example,  if $\alpha\in MU^{-n}(X)$ is represented by a smooth fiber bundle of closed smooth stably complex manifolds $E\longrightarrow X$ with $\dim E-\dim X=n$ and
$\beta\in MU_m(X)$ can be represented by a smooth map $f: M\longrightarrow X$, then the Kronecker pairing $\langle\alpha, \beta\rangle\in \Omega_{m-n}^U$
is the bordism class of the pull-back $\widetilde{f}^*(E)=E\times_X M$ as shown in the following diagram
\[
\begin{CD}
 E\times_X M   @>\widetilde{f}>>   E \\
    @VVV       @VVV \\
M @>f>> X \\
\end{CD}
\]
(cf.~\cite[D.3.4, Appendix D]{Panov}).

\subsection{The universal toric genus}
Based upon the works of tom Dieck, Krichever and L\"offler (see~\cite{tom Dieck1, tom Dieck3, Krichever, Loffler}),
 Buchstaber--Panov--Ray in \cite{Buch} defined the {\em universal toric genus} $\Phi$ in a geometric manner, which is a ring homomorphism from the geometric unitary $T^k$-bordism ring to the complex cobordism ring of the classifying space $BT^k$
\[
\Phi:\Omega_{*}^{U,T^k}\longrightarrow MU^{*}(BT^k).
\]
Note that, as an $\Omega_*^U$-algebra, $MU^{*}(BT^k)$ is isomorphic to $\Omega_*^U[[u_1, \cdots, u_k]]$ where $u_i$ is the cobordism Chern class $c_1^{MU}(\xi_i)$ of the conjugate Hopf bundle $\xi_i$ over the $i$-th factor of $BT^k$ for $i=1, ..., k$, so $MU^{*}(BT^k)$ can be replaced by $\Omega_*^U[[u_1, \cdots, u_k]]$.
 \vskip .1cm
 Let $[M]_{T^k}\in \Omega_{n}^{U,T^k}$ be an element represented by a closed unitary $T^k$-manifold $M$. Buchstaber--Panov--Ray showed that  $\Phi([M]_{T^k})$ can be defined to be the cobordism class of the complex oriented map $\pi:ET^k\times_{T^k}M \longrightarrow BT^k$.  More precisely, choose a $T^k$-equivariant embedding $i: M\hookrightarrow V$ into a unitary $T^k$-representation space $V$. Then the Borelification of $i$ gives a complex-oriented map
 $$\pi_l: ET^k(l)\times_{T^k} M\hookrightarrow ET^k(l)\times_{T^k}V\longrightarrow BT^k(l)$$
 which determines a cobordism class $\alpha_l$ in $MU^{-n}(BT^k(l))$, where $BT^k(l)=({\Bbb C}P^l)^k$ and $ET^k(l)=(S^{2l+1})^k$.  Since
 $BT^k=\cup_l BT^k(l)$ and $ET^k=\cup_l ET^k(l)$, these cobordism classes $\alpha_l$ form an inverse system. This produces a class $\alpha=\varprojlim \alpha_l$
 in $MU^{-n}(BT^k)$, which is represented geometrically by the complex-oriented map $$\pi: ET^k\times_{T^k} M\hookrightarrow ET^k\times_{T^k}V\longrightarrow BT^k.$$ Then $\Phi([M]_{T^k})$ is defined as this limit $\alpha$.
 \vskip .1cm
The following  result is essentially due to  L\"offler~\cite{Loffler} and Comeza$\widetilde{\text n}$a~\cite{C}  as noted by Hanke in~\cite{Hanke}.
\begin{prop}
The ring homomorphism $\Phi:\Omega_{*}^{U,T^k}\longrightarrow MU^{*}(BT^k)$ is injective.
\end{prop}
\begin{coro}\label{tgcoro}
If $M$ is a closed unitary $T^k$-manifold, then $M$ is null-bordant in $\Omega_*^{U,T^k}$ if and only if  the complex-oriented map
$\pi:ET^k\times_{T^k}M \longrightarrow BT^k$ represents the zero element in $MU^*(BT^k)$.
\end{coro}

\subsection{Equivariant Chern classes and equivariant Chern numbers}
Let $M$ be a closed unitary $T^k$-manifold. Then applying the Borel construction to the stable tangent bundle $TM$ of $M$ gives a vector bundle $ET^k\times_{T^k}TM$  over $ET^k\times_{T^k} M$.
\begin{defnm}
{\rm The {\em total equivariant cohomology Chern class} of $M$ is defined to be the total Chern class
of the vector bundle $ET^k\times_{T^k}TM$ over $ET^k\times_{T^k} M$, i.e.}
\begin{eqnarray*}
    c^{T^k}(M): &=& c(ET^k\times_{T^k}TM) \\
      &=&1+ c_1(ET^k\times_{T^k}TM)+c_2(ET^k\times_{T^k}TM)+
\cdots \\
      &=&1+ c^{T^k}_1(M)+c_2^{T^k}(M)+\cdots \in H^*(ET^k\times_{T^k} M)=H^*_{T^k}(M).
  \end{eqnarray*}
\end{defnm}

Let $p:M\longrightarrow pt$ be the constant map, and let  $p_!^{T^k}: H^*(ET^k\times_{T^k}M)\longrightarrow H^*(BT^k)$ be the equivariant Gysin map  induced by $p$.   Then,
the {\em integral equivariant cohomology Chern numbers} of $M$ are defined to be:
\begin{eqnarray*}
c_\omega^{T^k}[M]_{T^k}: =  p_!^{T^k}(c_\omega^{T^k}(M)),
\end{eqnarray*}
each of which is a homogeneous  polynomial of degree $2|\omega|-\dim M$ in the polynomial ring $H^*(BT^k)={\Bbb Z}[x_1, ..., x_k]$ with $\deg x_i=2$, where
$\omega=(i_1,  i_2,  \cdots, i_s)$ is a partition of $|\omega|=i_1+i_2+\cdots+i_s$, and $c_\omega^{T^k}(M)$ means the product $c_{i_1}^{T^k}(M)c_{i_2}^{T^k}(M)\cdots
c_{i_s}^{T^k}(M)$.
 \vskip .1cm
Note that the map $\pi: ET^k\times_{T^k}M\longrightarrow BT^k$ is the Borelification of  $p:M\longrightarrow pt$, so $p_!^{T^k}$ is often replaced by the Gysin
map $\pi_!$  in this paper.

\section{Proof of main theorem and equivariant unoriented bordism}\label{proof}

\subsection{Determine equivariant cohomology Chern numbers by ordinary Chern numbers}
Consider the equivariant cohomology Chern numbers $c_\omega^{T^k}[M]_{T^k}\in H^*(BT^k;\mathbb{Z})$
for the partition $\omega=(i_1,  i_2,  \cdots, i_s)$ of weight $|\omega|=i_1+i_2+\cdots+i_s$.
Since $H^*(BT^k;\Z)\cong \Z[x_1,x_2, \cdots, x_k]$ with $\deg x_i=2$, the equivariant Chern number $c_\omega^{T^k}[M]_{T^k}$ admits the form:
\[
c_\omega^{T^k}[M]_{T^k}=\sum_{J} n^{\omega}_J x^J,
\]
where the $k$-tuples of natural numbers $J=(j_1, j_2, ..., j_k)$ ranges over all $k$-partitions of weight $|\omega|-n$
and $x^J$ means $x_1^{j_1}x_2^{j_2}\cdots$ $x_k^{j_k}$, 
and the coefficient $n^{\omega}_J\in {\Bbb Z}$. It follows that to determine the equivariant cohomology Chern numbers $c_\omega^{T^k}[M]_{T^k}$ is equivalent to determine the coefficient $n^{\omega}_J$ for each $J=(j_1, j_2, ..., j_k)$. 

The coefficient $n^{\omega}_J$ is closely related to the ordinary Chern numbers of certain manifolds.
Indeed, for $J=(j_1, j_2, ..., j_k)$, let
$(\prod_{i=1}^{k} S^{2j_i+1})\times_{T^k} M$ be the pull-back of the natural inclusion $\prod_{i=1}^{k}\C P^{j_i} \hookrightarrow BT^k$: 
\[
\begin{CD}
(\prod_{i=1}^{k} S^{2j_i+1})\times_{T^k} M@>\widetilde{f_J}>>   ET^k\times_{T^k}M \\
@V \pi_J VV       @V \pi VV \\
\prod_{i=1}^{k}\C P^{j_i} @>f_J>> BT^k \\
\end{CD}
\]
First, we find that:
\begin{lem} 
$$n_J^{\omega}x^J=(\pi_J)_!(c_{\omega}((\prod_{i=1}^{k} S^{2j_i+1})\times_{T^k} TM)),$$
where $(\pi_J)_!$ is the Gysin map induced by $\pi_J$.
\end{lem}
\begin{proof}
By the definition of equivariant Chern numbers, 
\[c_\omega^{T^k}[M]_{T^k}=\pi_!(c_{\omega}(ET^k\times_{T^k}TM))=\sum_{J} n^{\omega}_J x^J\in H^*(BT^k;\Z).
\]
Since $(\prod_{i=1}^{k} S^{2j_i+1})\times_{T^k} TM=\tilde{f_J}^*(ET^k\times_{T^k}TM)$
and Gysin maps are commutative in the pull-back square, i.e, $f_J^*\pi_!=(\pi_J)_!\tilde{f_J}^*$, then we have:
\[
\begin{array}{ccl}
n_J^{\omega}x^J	 & = & f_J^*\pi_!(c_{\omega}(ET^k\times_{T^k}TM)) \\ 
	& = & (\pi_J)_!\tilde{f_J}^*c_{\omega}(ET^k\times_{T^k}TM) \\ 
	&=  & (\pi_J)_!(c_{\omega}((\prod_{i=1}^{k} S^{2j_i+1})\times_{T^k} TM)).
\end{array} 
\]
\end{proof}

Let $\omega=(i_1,  i_2,  \cdots, i_s)$ and $\omega'=(i'_1,  i'_2,  \cdots, i'_s)$ be two $s$-partitions, we say $\omega'\leqslant\omega$ provided $i'_j\leqslant i_j $ for each $j$ and we say $\omega'<\omega$ if $\omega'\leqslant\omega$ and $\omega'\neq\omega$. We also denote $\omega\pm\omega'$ the partition $(i_1\pm i'_1,  i_2\pm i_2',  \cdots, i_s\pm i'_s)$.

For each partition $\omega$ and each $k$-partition $J$, the corresponding coefficient $n_{J}^{\omega}$ is determined by the ordinary Chern number $c_{\omega}[(\prod_{i=1}^{k} S^{2j_i+1})\times_{T^k} M]$ and the coefficients $n_{J'}^{\omega'}$ with $J'<J,\omega'<\omega$.  More precisely, we have:
\begin{prop}\label{mainprop}
When dim$M>0$, one has:
\[
c_{\omega}[(\prod_{i=1}^{k} S^{2j_i+1})\times_{T^k} M]=n_{J}^{\omega}+
\sum_{
	\substack{
\omega'+\omega''=\omega	\\ 
	\omega'<\omega,\omega''<\omega,J'<J
	}
	}n_{J'}^{\omega'}m^{\omega''}_{(J-J')},
\]
where $m_{J-J'}$ is the coefficient of $x^{J-J'}$ in the Chern class:
\[
c_{\omega''}({\prod_{i=1}^{k}\C P^{j_i} })\in H^*(\prod_{i=1}^{k}\C P^{j_i} ;\Z)=\Z [x_1,x_2,\cdots,x_k]/(x_1^{j_1+1},\cdots,x_k^{j_k+1}).
\]
\end{prop}
\begin{proof}
The stable tangent bundle of $(\prod_{i=1}^{k} S^{2j_i+1})\times_{T^k} M$ admits a decomposition:
\[
T((\prod_{i=1}^{k} S^{2j_i+1})\times_{T^k}M) =((\prod_{i=1}^{k} S^{2j_i+1})\times_{T^k} TM) \oplus \pi_J^{*}T(\prod_{i=1}^{k}\C P^{j_i}).
\]
By Cartan formula, we have:
\[
c_{\omega}((\prod_{i=1}^{k} S^{2j_i+1})\times_{T^k} M)=c_{\omega}((\prod_{i=1}^{k} S^{2j_i+1})\times_{T^k} TM)+
\pi_J^{*}c_{\omega}(\prod_{i=1}^{k}\C P^{j_i})
\]
\[
 +\sum_{\substack{
 		\omega'+\omega''=\omega	\\ 
 		\omega'<\omega,\omega''<\omega
 		}}c_{\omega'}((\prod_{i=1}^{k} S^{2j_i+1})\times_{T^k} TM)\pi_J^{*}c_{\omega''}(\prod_{i=1}^{k}\C P^{j_i}).
\]
Taking $(\pi_J)_!$ to both sides, and we obtain:
\[
\begin{array}{ll}
&c_{\omega}[(\prod_{i=1}^{k} S^{2j_i+1})\times_{T^k} M]x^J  \ \ \ \  \ \ \small{\text{$(\pi_J)_!$ is isomorphic on the top cohomology}}         \\
=  & (\pi_J)_!c_{\omega}((\prod_{i=1}^{k} 
S^{2j_i+1})\times_{T^k} M) \\ 
 = &(\pi_J)_!(c_{\omega}((\prod_{i=1}^{k} S^{2j_i+1})\times_{T^k} TM))+ (\pi_J)_!(\pi_J^{*}c_{\omega}(\prod_{i=1}^{k}\C P^{j_i}))\\ 
+  & \sum_{\substack{	
	\omega'+\omega''=\omega	\\ 
	\omega'<\omega,\omega''<\omega
	} }(\pi_J)_!(c_{\omega'}((\prod_{i=1}^{k} S^{2j_i+1})\times_{T^k} TM)\pi_J^{*}c_{\omega''}(\prod_{i=1}^{k}\C P^{j_i}))\\
 =   & n_J^{\omega}x^J+(\pi_J)_!(1)c_{\omega}(\prod_{i=1}^{k}\C P^{j_i})   \ \ \ \   \text{ $(\pi_J)_!(1)$=0, when dim$M$>0 } \\ 
 +  &  \sum_{\substack{
 		\omega'+\omega''=\omega	\\ 
 		\omega'<\omega,\omega''<\omega,J'<J
 		} }c_{\omega''}(\prod_{i=1}^{k}\C P^{j_i}) (\pi_J)_!(c_{\omega'}((\prod_{i=1}^{k} S^{2j_i+1})\times_{T^k} TM)).\\
\end{array} 
\]
Note that:
\[
\begin{array}{lll}
(\pi_J)_! (c_{\omega'}((\prod_{i=1}^{k} S^{2j_i+1})\times_{T^k} TM))&=& f_J^*(\pi_!c_{\omega'}(ET^k\times_{T^k} TM))\\
&=& f_J^*((\sum  n_{J'}^{\omega'})x^{J'})\\
&=& \sum  n_{J'}^{\omega'} x^{J'}, 
\end{array}
\]
because $f_J^*(x^{J'})=x^{J'}$ when $J'\leqslant J$. 
Denote $c_{\omega''}(\prod_{i=1}^{k}\C P^{j_i})=\sum_{J''}m^{\omega''}_{J''}x^{J''}$ and we have:
\[
\begin{array}{ll}
&c_{\omega}[(\prod_{i=1}^{k} S^{2j_i+1})\times_{T^k} M]x^J\\
=   & n_J^{\omega}x^J+ \sum_{\substack{
	\omega'+\omega''=\omega	\\ 
	\omega'<\omega,\omega''<\omega,J'<J
	} } (\sum  n_{J'}^{\omega'} x^{J'})(\sum_{J''}m^{\omega''}_{J''}x^{J''}) \\
=&  (n_J^{\omega}+\sum_{\substack{
	\omega'+\omega''=\omega	\\ 
	\omega'<\omega,\omega''<\omega,J'<J
	} } n_J^{w'}m^{\omega''}_{J-J'})x^J.
\end{array} 
\]
Therefore, by comparing the coefficients of both sides, one has the equation: 
\[
c_{\omega}[(\prod_{i=1}^{k} S^{2j_i+1})\times_{T^k} M]=n_{J}^{\omega}+
\sum_{
	\substack{
		\omega'+\omega''=\omega	\\ 
		\omega'<\omega,\omega''<\omega,J'<J
} 
}n_{J'}^{\omega'}m^{\omega''}_{(J-J')}.
\]
\end{proof}
Clearly, one has
\begin{coro}
$c_{\omega}[(\prod_{i=1}^{k} S^{2j_i+1})\times_{T^k} M]\in (n_{J'}^{\omega'})\subset \Z$ with $\omega'\leqslant\omega,J'=(j'_1,\cdots,j'_k)\leqslant J$, 
where $(n_{J'}^{\omega'})$ is the ideal in $\Z$ generated by the integers $n_{J'}^{\omega'}$.
	
\end{coro}

On the contrary, when dim$M=n>0$, the coefficient $n_J^{\omega}$ of the equivariant Chern number $c_\omega^{T^k}[M]_{T^k}$ is determined by the ordinary Chern numbers $c_{\omega'}[(\prod_{i=1}^{k} S^{2j'_i+1})\times_{T^k} M]$ of the manifolds $(\prod_{i=1}^{k} S^{2j_i'+1})\times_{T^k} M$ with partition $\omega'\leqslant\omega$ and $J'\leqslant J$. 
Indeed, we have:
\begin{prop}\label{mainequation}
$n_{J}^{\omega}\in (c_{\tilde{\omega}}[(\prod_{i=1}^{k} S^{2\tilde{j_i}+1})\times_{T^k} M])\subset \Z$ with $\tilde{\omega}\leqslant\omega,\tilde{J}\leqslant J$, 
where the ideal $(c_{\tilde{\omega}}[(\prod_{i=1}^{k} S^{2\tilde{j_i}+1})\times_{T^k} M])$ in $\Z$ is generated by the integers $c_{\tilde{\omega}}[(\prod_{i=1}^{k} S^{2\tilde{j_i}+1})\times_{T^k} M]$.
\end{prop}	
\begin{proof}
Note that when the weight $|\omega|<n$, $n_J^{\omega}=0$ for all $J$. 
When the weight $|\omega|=n$, the corresponding $k$-partition is $J_0=(0,\cdots,0)$ and $n^{\omega}_{J_0}=c_{\omega}[(\prod_{i=1}^{k} S^{2j_i+1})\times_{T^k} M]$. 

When $|\omega|>n$, by Proposition \ref{mainprop}, we see 
$$n_{J}^{\omega}
-c_{\omega}[(\prod_{i=1}^{k} S^{2j_i+1})\times_{T^k} M]\in (n_{J'}^{\omega'})\subset \Z$$
with $\omega'<\omega,J'<J$. When $|\omega'|<|\omega|$, by the induction of the weight $|\omega|$, we obtain that 
$$(n_{J'}^{\omega'})\subset (c_{\omega''}[(\prod_{i=1}^{k} S^{2j''_i+1})\times_{T^k} M])\subset \Z,$$ 
with 
$\omega''\leqslant\omega',J''\leqslant J'$. Hence, we have:
$$n_{J}^{\omega}\in 
c_{\omega}[(\prod_{i=1}^{k} S^{2j_i+1})\times_{T^k} M] + (n_{J'}^{\omega'})\subset (c_{\tilde{\omega}}[(\prod_{i=1}^{k} S^{2\tilde{j_i}+1})\times_{T^k} M])\subset   \Z$$
with $\tilde{\omega}\leqslant\omega,\tilde{J}\leqslant J$.

\end{proof}
\begin{thm}\label{maintech}
The equivariant Chern number $c_\omega^{T^k}[M]_{T^k}=0$ for all the partition $\omega$ if and only if the ordinary Chern number $c_{\omega}[(\prod_{i=1}^{k} S^{2j_i+1})\times_{T^k} M]=0$ for all $\omega$ and all $J=(j_1, j_2, ..., j_k)$.
\end{thm}
\begin{proof}
When dim$M=0$, $M$ is the disjoint union of some isolated points with trivial torus action, and the statement certainly holds. Now we assume dim$M>0$.
Suppose $c_\omega^{T^k}[M]_{T^k}=0$ for all $\omega$, then $n_{\tilde{J}}^{\tilde{\omega}}=0$ for all $\tilde{\omega}$ and $\tilde{J}$. By Proposition~\ref{mainequation}, $c_{\omega}[(\prod_{i=1}^{k} S^{2j_i+1})\times_{T^k} M]\in (n_{\tilde{J}}^{\tilde{\omega}})=(0)\subset \Z$. The inverse direction is similar.	
\end{proof}

\subsection{Proof of Theorem~\ref{main1}}\label{man}

Let $M$ be a closed unitary $T^k$-manifold. Since the universal toric genus
\[
\Phi:\Omega_{*}^{U,T^k}\longrightarrow MU^{*}(BT^k)
\]
is injective,
it is sufficient to show that the image $\Phi([M]_{T^k})=0$ if and only if
all integral equivariant cohomology Chern numbers  of $M$ vanish.
 
 \vskip .1cm
Geometrically, the image $\Phi([M]_{T^k})\in MU^*(BT^k)$ is represented by the map $\pi: ET^k\times_{T^k} M \longrightarrow BT^k.$
Since the Kronecker pairing
$$MU^*(BT^k)\otimes MU_*(BT^k)\longrightarrow MU_*.$$
defines an isomorphism
\[
MU^*(BT^k)\cong Hom_{MU_*}(MU(BT^k), MU_*),
\]
because $MU^*(BT^k)$ is a free $MU_*$-module in even dimensional generators. 

Therefore, the image $\Phi([M]_{T^k})$ is zero if and only if its evaluation with all the $MU_*$ generators of
$MU_*(BT^k)$  is zero.

According to \cite{AD,Kochman}, the natural inclusion maps $\C P^n\hookrightarrow BT^1$ are the $MU_*$-generators of $MU_*(BT^1)$, and therefore one can take the natural inclusion maps
$\prod_{i=1}^{k}\C P^{j_i}\hookrightarrow BT^k$ as $MU_*$-generators of $MU_*(BT^k)$. 

By Quillen's construction, the Kronecker pairing of $\Phi([M]_{T^k})$ and $ [\prod_{i=1}^{k}\C P^{j_i}\rightarrow BT^k]$
defines the bordism class of the manifold:
\[
<\Phi([M]_{T^k}), [\prod_{i=1}^{k}\C P^{j_i}\rightarrow BT^k]>=[(\prod_{i=1}^{k} S^{2j_i+1})\times_{T^k} M]\in MU_*.
\]
Therefore, $\Phi([M]_{T^k})$ is zero if and only if $[(\prod_{i=1}^{k} S^{2j_i-1})\times_{T^k} M]=0\in MU_*$ for all
the inclusions $\prod_{i=1}^{k}\C P^{j_i}\hookrightarrow BT^k$ for all $J=(j_1, j_2, ..., j_k)$.

On the other hand, by Theorem~\ref{maintech}, we see the equivariant Chern numbers $c_\omega^{T^k}[M]_{T^k}=0$ if and only if the ordinary Chern numbers $c_{\omega}[(\prod_{i=1}^{k} S^{2j_i-1})\times_{T^k} M]=0$ for all $\omega$ and all $J=(j_1, j_2, ..., j_k)$.

Thus, we conclude that the equivariant Chern number $c_\omega^{T^k}[M]=0$ for all $\omega$ if and only if
$\Phi([M])=0$. Hence, Theorem~\ref{main1} has been proved.

\

Together with the above arguments, we conclude:
\begin{thm}\label{first-thm}
Let $M$ be a closed unitary $T^k$-manifold. Then the following statements are equivalent.
\begin{enumerate}
\item[(1)] $M$ is null-bordant in $\Omega_*^{U, T^k}$.
\item[(2)]  All integral equivariant cohomology Chern numbers of $M$ vanish.
\item[(3)] For any partition $J=(j_1, ..., j_k)$ with each $j_i\geq 0$, $(\prod_{i=1}^k S^{2j_i+1})\times_{T^k}M$ is null-bordant in $\Omega_*^U$.
\item[(4)] The fibration $\pi: ET^k\times_{T^k}M\longrightarrow BT^k$ is null-cobordant in $MU^*(BT^k)$.
\end{enumerate}
\end{thm}

\subsection{Equivariant unoriented bordism} Our approach above can also be carried out in the case of equivariant unoriented bordism.
Let $\mathfrak{N}_*^{({\Bbb Z}_2)^k}$ (or $\Omega_*^{O,({\Z_2})^k}$) be the ring formed by the equivariant bordism classes of all unoriented closed smooth $({\Bbb Z}_2)^k$-manifolds, and let
$MO^*(X)$ be the unoriented (homotopic) cobordism ring of a topological space $X$, i.e.,
$$MO^*(X)=\lim_{l\longrightarrow\infty}[S^{l-*}\wedge X_+, MO(l)].$$
Then Quillen's geometric approach on $MU^*(X)$ can be carried out in the case of $MO^*(X)$, so that the Kronecker pairing
$$\langle\ ,\ \rangle: MO^*(X)\otimes MO_*(X)\longrightarrow MO_*\cong \mathfrak{N}_*$$
can be calculated in a geometric way similar to the unitary case as in Subsection~\ref{geo} (cf.~\cite[D.2.8, D.3.4, Appendix D]{Panov}), where $\mathfrak{N}_*$ is the nonequivariant Thom unoriented bordism ring.
In~\cite{tom Dieck2}, tom Dieck showed that there is also a monomorphism
$$\Phi_{\Bbb R}: \mathfrak{N}_*^{({\Bbb Z}_2)^k}\longrightarrow MO^*(B({\Bbb Z}_2)^k).$$
It is well-known that $MO^*(B({\Bbb Z}_2)^k)=\varprojlim MO^*(B({\Bbb Z}_2)^k(n))$ where $B({\Bbb Z}_2)^k(n)=({\Bbb R}P^n)^k$ and $B({\Bbb Z}_2)^k=\cup_nB({\Bbb Z}_2)^k(n)$. Then, applying the finite approximation method yields that for a class $[M]_{({\Bbb Z}_2)^k}\in  \mathfrak{N}_*^{({\Bbb Z}_2)^k}$, the image $\Phi_{\Bbb R}([M]_{({\Bbb Z}_2)^k})$ is geometrically represented by the fibration $\pi: E({\Bbb Z}_2)^k\times_{({\Bbb Z}_2)^k}M \longrightarrow B({\Bbb Z}_2)^k$.
\begin{thm}\label{unoriented}
Let $M$ be a closed smooth $({\Bbb Z}_2)^k$-manifold. Then the following statements are equivalent.
\begin{enumerate}
\item[(1)] $M$ is null-bordant in $\mathfrak{N}_*^{({\Bbb Z}_2)^k}$.
\item[(2)]  All equivariant  Stiefel--Whitney numbers of $M$ vanish.
\item[(3)] For any partition $I=(j_1, ..., j_k)$ with each $j_i\geq 0$, $(\prod_{i=1}^k S^{j_i+1})\times_{({\Bbb Z}_2)^k}M$ is null-bordant in $\mathfrak{N}_*$.
\item[(4)] The fibration $\pi:  E({\Bbb Z}_2)^k\times_{({\Bbb Z}_2)^k}M \longrightarrow B({\Bbb Z}_2)^k$ is null-cobordant in $MO^*(B({\Bbb Z}_2)^k)$.
\end{enumerate}
\end{thm}
\begin{proof}
The proof follows closely that of Theorem~\ref{first-thm} in a very similar way. We would like to leave it as an exercise to the reader.
\end{proof}

\begin{rem}
 Theorem~\ref{unoriented} tells us that $({\Bbb Z}_2)^k$-equivariant unoriented  bordism is determined by $({\Bbb Z}_2)^k$-equivariant  Stiefel--Whitney numbers. This result is essentially due to  tom Dieck with an argument of involving the Boardman map $$B: MO^*(B(\Z_2)^k)\longrightarrow H^*(B(\Z_2)^k;\Z_2)\widehat{\otimes}H_*(BO;\Z_2)$$ (see \cite{tom Dieck2}). Here we exactly replace the use of the Boardman map by the Kronecker pairing.
\end{rem}

It is well-known that as an $\mathfrak{N}_*$-algebra, $MO^*(B({\Bbb Z}_2)^k)$ is isomorphic to $\mathfrak{N}_*[[v_1, ..., v_k]]$, where $v_i$ denotes the first cobordism Stiefel--Whitney class of the canonical line bundle over the $i$-th factor of $B({\Bbb Z}_2)^k$ for $i=1, ..., k$.  Thus, in the sense of Buchstaber--Panov--Ray \cite{Buch},  we can also define {\em $({\Bbb Z}_2)^k$-equivariant universal genus}, which is an $\mathfrak{N}_*$-algebra homomorphism, as follows:
$$\Phi_{\Bbb R}: \mathfrak{N}_*^{({\Bbb Z}_2)^k}\longrightarrow \mathfrak{N}_*[[v_1, ..., v_k]],$$
by
$\Phi_{\Bbb R}([M]_{({\Bbb Z}_2)^k})=\sum_{\omega=(j_1, ..., j_k)}g_\omega^{\Bbb R}(M)v^\omega$, where   $v^\omega=v_1^{j_1}\cdots v_k^{j_k}$ for $\omega=(j_1, ..., j_k)$ with $j_i\geq 0$, and $g_\omega^{\Bbb R}(M)\in \mathfrak{N}_*$ with $\dim g_\omega^{\Bbb R}(M)=|\omega|+\dim M$.
\vskip .1cm
 The determination of coefficients $g_\omega^{\Bbb R}(M)$ depends upon the choices of a dual bases in $MO_*(B({\Bbb Z}_2)^k)$ to the basis $\{v^\omega\}$ in $MO^*(B({\Bbb Z}_2)^k)$.
 Buchstaber--Panov--Ray's approach  in the unitary case provides us much more insight on the determination of coefficients $g_\omega^{\Bbb R}(M)$.
 In our case, since it was shown in~\cite[Corollary 6.4]{LT} that any real Bott manifold is null-bordant in $\mathfrak{N}_*$, we may employ the real Bott manifolds to  realize a dual basis in $MO_*(B({\Bbb Z}_2)^k)$ of the basis $\{v^\omega\}$ in $MO^*(B({\Bbb Z}_2)^k)$, so that the coefficients $g_\omega^{\Bbb R}(M)$ can be represented by manifolds $G_\omega^{\Bbb R}(M)=(S^1)^\omega\times_{({\Bbb Z}_2)^\omega}M$, where the action of $({\Bbb Z}_2)^\omega=({\Bbb Z}_2)^{j_1}\times\cdots\times ({\Bbb Z}_2)^{j_k}$ on $(S^1)^\omega=(S^1)^{j_1}\times\cdots\times (S^1)^{j_k}$ is coordinatewise, and on $M$ via the representation $(g_{1,1}, ..., g_{1, j_1}; ...; g_{k,1}, ..., g_{k, j_k}) \longmapsto (g_{1, j_1}, ..., g_{k, j_k})$. The argument is close to that of the unitary case in \cite[\S 9.2]{Panov}. 



\end{document}